\documentclass[12pt,reqno]{amsart}

\usepackage{amssymb,bm,mathrsfs}

\usepackage{enumerate}
\usepackage{tikz}
\usepackage[centering,width=6in]{geometry}
\makeatother

\usepackage{calc}
\usepackage{accents}

\newtheorem{theoremvar}{Theorem}
\newenvironment{theorem*}[1]
 {\renewcommand{\thetheoremvar}{#1}\theoremvar}
 {\endtheoremvar}

 \newtheorem{corollaryvar}{Corollary}
\newenvironment{corollary*}[1]
 {\renewcommand{\thecorollaryvar}{#1}\corollaryvar}
 {\endtheoremvar}

\newtheorem{theorem}{Theorem}[section]
\newtheorem{lemma}[theorem]{Lemma}
\newtheorem{proposition}[theorem]{Proposition}
\newtheorem{corollary}[theorem]{Corollary}

\newtheorem{conjecture}[theorem]{Conjecture}
\theoremstyle{definition}
\newtheorem{definition}[theorem]{Definition}
\newtheorem{example}[theorem]{Example}
\newtheorem{remark}[theorem]{Remark}
\newcommand{\R}{\mathbb{R}}

\newcommand{\Z}{\mathbb{Z}}
\newcommand{\N}{\mathbb{N}}

\newcommand{\G}{\mathbb{G}}

\usepackage[framemethod=TikZ]{mdframed}
\theoremstyle{remark}
\usetikzlibrary{shadows}
\usepackage{amsthm}

\usepackage[colorlinks,pdfstartview=FitH]{hyperref}

\hypersetup{
    colorlinks,
    citecolor=blue,
    filecolor=black,
    linkcolor=red,
    urlcolor=black
}

\usepackage{graphicx,amsmath,amssymb}

\newcommand\rsetminus{\mathbin{\mathpalette\rsetminusaux\relax}}
\newcommand\rsetminusaux[2]{\mspace{-4mu}
  \raisebox{\rsmraise{#1}\depth}{\rotatebox[origin=c]{-20}{$#1\smallsetminus$}}
 \mspace{-4mu}
}
\newcommand\rsmraise[1]{%
  \ifx#1\displaystyle .8\else
    \ifx#1\textstyle .8\else
      \ifx#1\scriptstyle .6\else
        .45%
      \fi
    \fi
  \fi}

\begin{document}

\title[Topological Erd\H{o}s similarity conjecture and Smz sets]{Topological Erd\H{o}s similarity conjecture and strong measure zero sets}
% \title[Erd\H{o}s similarity conjecture, strong measure zero sets, and strongly meager sets]{Erd\H{o}s similarity conjecture, strong measure zero sets, and strongly meager sets}

\author{Yeonwook Jung}
\address{University of California, Irvine, Department of Mathematics, Rowland Hall, Irvine, CA 92697}
\curraddr{}
\email{yeonwoj1@uci.edu}
\thanks{}

%    author two information
\author{Chun-Kit Lai}
\address{San Francisco State University Department of Mathematics, 1600 Holloway Ave, San Francisco, CA 94132
}
\curraddr{}
\email{cklai@sfsu.edu}
\thanks{}

%    \subjclass is required.
\subjclass[2010]{Primary }

% \date{}

%\dedicatory{In preparation for publication.}

%    Abstract is required.
\begin{abstract}

We solve a topological variant of the Erd\H{o}s similarity conjecture posed by Gallagher, Lai, and Weber. Particularly, we prove that the existence of an uncountable subset of the real line that is universal in the family of all dense $G_\delta$ sets is independent of the usual axioms of set theory. Such sets are shown to correspond to strong measure zero sets, which were first introduced by Borel. Applying the duality principle between measure and category to this result, we derive a conjecture that is strong enough to prove the Erd\H{o}s similarity conjecture for perfect sets.

\end{abstract}

\maketitle
%\tableofcontents

% ?? change start of section headings a bit.

\section{Introduction}

% ?? restricting to perfect sets is reasonable e.g. CH. This might imply that original Erdos similarity conjecture might be amenable to set theoretic tools more than it seems.

% ?? But its not even for uncountable sets. We can answer for perfect sets, if it is shown to be true. So focus on topological one.

A subset $X$ of $\R$ is said to be \textbf{universal} in a family $\mathcal{F}$ of subsets of $\R$ if every $F\in \mathcal{F}$ contains an affine copy of $X$, i.e., $\lambda X + t\subset F$ for some $\lambda \in \R\setminus \{0\}$ and for some $t\in \R$. A subset $X$ of $\R$ is said to be \textbf{measure universal} if it is universal in the family of all Lebesgue measurable subsets of $\R$ that have positive Lebesgue measure. It is a consequence of the Lebesgue density theorem that finite sets are universal \cite{STE20}. Based on this observation, Erd\H{o}s conjectured that there is no infinite measure universal set, referred to as the Erd\H{o}s similarity conjecture \cite{erdos-1974}.

\begin{conjecture}[Erd\H{o}s]
    \label{conj:original-erdos-similarity-conjecture}
    No infinite subset of the real line is measure universal.
\end{conjecture}

% Only specific classes of infinite sets are known to be not measure universal, rendering the Erd\H{o}s similarity an open problem.
% In the following, we review its past developments very briefly. For more detailed review, see \cite{svetic}.
% Sumsets of infinite sets are another classes of infinite sets that were studied for their measure universality.

To prove the Erd\H{o}s similarity conjecture, it suffices to show that all sequences that monotonically decreases and converges to zero are not universal. Among such sequences, sublacunary sequences $\{a_n\}$, i.e., those that satisfy $\lim_{n\to\infty} \frac{a_{n+1}}{a_n}=1$, are known to be not measure universal \cite{Eig85,Fal84}, but exponentially decaying sequences, such as $\{2^{-n}\}$, remain unanswered. Also, sumsets of the form $\{a_n\}+\{b_n\}+\{c_n\}$ or $\{2^{-n^\alpha}\}+\{2^{-n^\alpha}\}$ are not measure universal \cite{Bou87,Kol97}. A comprehensive survey on the Erd\H{o}s similarity conjecture containing results published before year 2000 can be found in \cite{svetic}. A review on some of the recent progress on the Erd\H{o}s similarity conjecture can be found in \cite{Jung-Lai-Mooroogen}.

\medskip

Despite these results, arbitrary countable sets still remain elusive in terms of measure universality. In particular, the Erd\H{o}s similarity conjecture remains open for arbitrary uncountable sets, and even for perfect sets.

\medskip

In part of the effort to further understand the Erd\H{o}s similarity conjecture, Gallagher, Lai, and Weber recently considered a topological variant of the Erd\H{o}s similarity conjecture by defining topological universality as follows \cite{gallagpher-lai-weber}. A subset $X$ of $\R$ is said to be \textbf{topologically universal} if it is universal in the family of all dense $G_\delta$ subsets of $\R$. They observed that countable sets are topologically universal by the Baire category theorem and proved that perfect sets are not topologically universal. In particular, they completed the analogy to the Erd\H{o}s similarity conjecture by stating the following conjecture.

\begin{conjecture}[Gallagher-Lai-Weber]\label{conj:gallagher-lai-weber}
    \label{topological-erdos-similarity-conjecture-r}
    No uncountable subset of the real line is topologically universal.
\end{conjecture}

In this paper, we give a complete solution to Conjecture \ref{conj:gallagher-lai-weber} on $\R$ and more generally on locally compact Polish groups. We then discuss its implication to the original Erd\H{o}s similarity conjecture. The main results and the outline of this paper are given below.

\medskip

In Section \ref{section:2}, we prove that Conjecture \ref{conj:gallagher-lai-weber} is independent of the standard axioms of set theory, namely, Zermelo-Fraenkel axioms of set theory with axiom of choice (\textsf{ZFC}). We achieve this by proving that {\it the  topologically universal sets are precisely the strong measure zero sets}, which is defined as the following. A subset $X$ of the real line is a {\bf strong measure zero set} if for every sequence $\{\epsilon_n\}_{n=1}^\infty$ of positive real numbers, there exists a sequence of intervals $
\{I_n\}_{n=1}^\infty$ such that $X\subset \cup_{n=1}^\infty I_n$ and $|I_n|< \epsilon_n$. Note that all countable sets are strong measure zero sets. \textbf{Borel conjecture} is the statement that all strong measure zero sets are countable \cite{borel}. It is known that the Borel conjecture is independent of \textsf{ZFC}: Sierpi\'{n}ski \cite{sierpinski} proved the consistency of the negation of the Borel conjecture, and Laver \cite{laver} proved the consistency of the Borel conjecture. This concludes that {\it Conjecture \ref{conj:gallagher-lai-weber} is independent of \textsf{ZFC}}.

\medskip

The fact that Conjecture \ref{conj:gallagher-lai-weber} is independent of \textsf{ZFC} readily generalizes to a more general space of locally compact Polish groups. In particular, this includes $\R^d$ as a special case. Similar to the case of the real line, we prove that the existence of uncountable topologically universal set in any locally compact Polish group is independent of \textsf{ZFC}. This is proved in Section \ref{section:3}. Although the result for the real line can be deduced from the result for locally compact Polish groups, we still choose to present the proof for the real line since the proof is clearer without the additional definitions and the caveats for working in locally compact Polish groups.

\medskip

In Section \ref{section:4}, our resolution of this topological variant of the Erd\H{o}s similarity conjecture is connected to the original Erd\H{o}s similarity conjecture in the following way. Under the principle of duality between measure and category, we obtain the following conjecture as a dual statement of our main result in section 2: a subset of the real line is universal in the family of all sets of full measure if and only if it is a strongly meager set. Since perfect sets are not strongly meager \cite{erdos-kunen-mauldin}, if this conjecture is true, then perfect sets are not universal in sets of full measure, and thus, not measure universal. This would establish a stronger version of the Erd\H{o}s similarity conjecture for perfect sets.

\medskip

\subsection{Notation and basic definitions}
Unless stated otherwise, we work in the Zermelo–Fraenkel set theory with axiom of choice (\textsf{ZFC}). For subsets $A, B$ of a group, $A+B=\{a+b:a\in A, b\in B\}$ denotes the algebraic sum or the Minkowski sum of $A$ and $B$. If $X$ is a subset of a vector space, then $\lambda X = \{ \lambda x : x\in X\}$, where $\lambda$ is a scalar, denotes the scaling of $X$ by $\lambda$. In a topological space, a {\bf meager set} is a countable union of nowhere dense sets. A {\bf nowhere dense set} is a set whose closure has empty interior. A set is {\bf comeager} if its complement is meager. A topological space is called a \textbf{Baire space} if every meager set has no interior, or equivalently, every comeager set is dense. A {\bf $\sigma$-ideal} is a collection of sets that contains the empty set and is closed under taking subsets and countable unions. 

% {\color{red} Define Baire space}.

\medskip
\section{Topologically universal sets in $\R$}\label{section:2}
In this section, we classify topologically universal sets in $\R$ and answer the Conjecture \ref{conj:gallagher-lai-weber}.

\begin{lemma} \label{lem:1}
    In a Baire space, every dense $G_\delta$ set contains a comeager set and every comeager set contains a dense $G_\delta$.
     % that is, for all comager set $F\subset \R$, there exist $\lambda \neq 0$ and $t\in\R$ such that $\lambda A + t \subset F$.
\end{lemma}
\begin{proof}
Let $X$ be a Baire space. Let $G$ be a dense $G_\delta$ subset of $X$. Then, $G=\cap_{n=1}^\infty G_n$ for some open subsets $G_n$ of $X$, where each $G_n$ is dense since $\text{cl}(G_n)\supset \text{cl}(G) =X$. Note that $X\rsetminus G = X\rsetminus (\cap_{n=1}^\infty G_n)=\cup_{n=1}^\infty (X\rsetminus G_n)$, where $X\rsetminus G_n$ is nowhere dense since $\left(\text{cl}(X\rsetminus G_n)\right)^\circ = (X\rsetminus G_n)^\circ = X\rsetminus \text{cl}(G_n)= \varnothing$. Hence, $X\rsetminus G$ is a meager set and $G$ is a comeager set. Thus, $G$ contains a comeager set, namely, $G$ itself.

Next, let $H$ be a comeager subset of $X$. Then, $H=X\rsetminus M$ for some meager subset $M$ of $X$ and $M=\cup_{n=1}^\infty N_n$, where each $N_n$ is a nowhere dense subset of $X$. Since each $\text{cl}(N_n)$ is also a nowhere dense subset of $X$, $\cup_{n=1}^\infty \text{cl}(N_n)$ is a meager $F_\sigma$ set, implying that $X\rsetminus \cup_{n=1}^\infty \text{cl}(N_n)$ is a comeager $G_\delta$ set that is a subset of $H$. Since $X$ is a Baire space and comeager sets are dense in Baire spaces, $H$ contains a dense $G_\delta$ set.
\end{proof}

\begin{lemma}
    A subset $X$ of $\R$ is topologically universal if and only if $X$ is universal in the family of all comeager sets.
    \label{lem:2}
     % that is, for all comager set $F\subset \R$, there exist $\lambda \neq 0$ and $t\in\R$ such that $\lambda A + t \subset F$.
\end{lemma}
\begin{proof}
Let $X$ be a topologically universal subset of $\R$ and $H$ be a comeager subset of $\R$. By Lemma \ref{lem:1}, there exists a dense $G_\delta$ set $G$ such that $G\subset H$ and since $X$ is topologically universal, there is an affine copy $\lambda X +t$, for some $\lambda \neq 0$ and $t\in \R$, of $X$, such that $\lambda X + t \subset G$. Thus, $\lambda X +t \subset H$ and $X$ is universal in comeager sets. The other direction is obtained by the same proof upon exchanging comeager sets and $G_\delta$ sets.
\end{proof}

In order to make connection between topologically universal sets and strong measure zero sets in $\R$, we need the following key theorem by Galvin, Mycielski, and Solovay \cite{Galvin-Mycielski-Solovay-1973} (announced without proof). For a proof, see \cite{BJ95} and \cite{Galvin-Mycielski-Solovay-2017}.

\begin{theorem*}{A}[Galvin-Mycielski-Solovay]\label{thm:GMS}
    A subset $X$ of $\R$ is a strong measure zero set if and only if for every meager set $M$ in $\R$, $X+M\neq \R$.
\end{theorem*}

We notice that topological universality can also be phrased as a sum set problem similar to the Galvin-Mycielski-Solovay theorem (Theorem \ref{thm:GMS}).

\begin{lemma}
\label{lem:3}
 A subset $X$ of $\R$ is topologically universal if and only if for every meager set $M$ in $\R$, there exists $\lambda \in \R\rsetminus \{0\}$ such that $X+\lambda M \neq \R$.
\begin{proof}
We first prove the only if. Suppose $X$ is topologically universal subset of $\R$. Then, by Lemma \ref{lem:2}, for every meager set $M$ in $\R$, there exists $\lambda' \in \R\setminus \{0\}$ and $t\in \R$ such that $\lambda' X +t \subset M^c$, i.e., $(\lambda' X + t)\cap M =\varnothing$. This implies that for all $x\in X$ and $y\in M$, $\lambda' x+t \neq y$, i.e., $x-\frac{1}{\lambda '} y \neq -\frac{t}{\lambda '}$. Hence, $-\frac{t}{\lambda'}\not\in X-\frac{1}{\lambda '}M$, i.e., $X-\frac{1}{\lambda}M \neq \R$. Taking $\lambda = -\frac{1}{\lambda '}$, we have $X+\lambda M \neq \R$.

To see the if part, let $X$ be a subset of $\R$ and suppose that for each meager set $N$ in $\R$, there exists $\lambda' \in \R\setminus \{0\}$ such that $X+\lambda' N\neq \R$. Let $M$ be a meager set in $\R$. Then, there exists $a\in \R$ such that for all $x\in X$ and $y \in M$, $x+\lambda' y \neq a$, i.e., $\frac{1}{\lambda '}x-\frac{a}{\lambda'} \neq - y$. Thus, by letting $\lambda = \frac{1}{\lambda '}$ and $t=-\frac{a}{\lambda'}$, we have $(\lambda X+t) \cap M = \varnothing$, i.e., $\lambda X + t \subset M^c$. Therefore, $X$ is universal in comeager subsets of $\R$, and by Lemma \ref{lem:2}, $X$ is topologically universal.
\end{proof}
\end{lemma}

% \begin{corollary}
%     Every countable subset of $\R$ is topologically universal.
% \end{corollary}
% \begin{proof}
%     Let $X\subset \R$ be a countable set and $M\subset \R$ be a meager set. Since countable union of meager subsets is meager, $X+M$ is a meager set, and hence, $X+M\neq \R$. By Lemma \ref{lem}
% \end{proof}
% Since countable union of meager subsets is meager,  in any topological space

\begin{theorem}
    A subset of the real line is topologically universal if and only if it is a strong measure zero set.
    \label{thm:R-top-univ}
\end{theorem}
\begin{proof}
    We first show the if part. Suppose that $X$ is a strong measure zero set. By the Galvin-Mycielski-Solovay theorem (Theorem \ref{thm:GMS}), $X+M\ne\R$ for all meager sets $M$ in $\R$. We now take $\lambda = 1$ and apply Lemma \ref{lem:3} to conclude that $X$ is topologically universal.

Next, we show the only if part. Let $X$ in $\R$ be a topologically universal. Let $(\epsilon_n)_{n\in \N}$ be a sequence of positive reals. Define 
\begin{equation*}
G=\bigcap_{k=1}^{\infty}\bigcup_{n=1}^{\infty}\left(q_n-\frac{\epsilon_n}{k}, q_n+\frac{\epsilon_n}{k}\right),
\end{equation*}
where $(q_n)_{n\in\N}$ is an enumeration of the rational numbers. Since $G$ is a dense $G_\delta$ set, there exists $\lambda \neq 0$ and $t\in\mathbb{R}$ such that $\lambda X + t \subset G$. Let $k_0 = \lceil \frac{2}{|\lambda|} \rceil$. Then,
\begin{equation*}
\lambda X + t \subset \bigcup_{n=1}^{\infty}\left(q_n-\frac{\epsilon_n}{k_0}, q_n+\frac{\epsilon_n}{k_0}\right),
\end{equation*}
Denoting $I_n=\frac{1}{\lambda}\left(q_n-\frac{\epsilon_n}{k_0}-t, q_n+\frac{\epsilon_n}{k_0}-t\right)$, we have $X\subset\bigcup\limits_{n\in\N}I_n$ with $|I_n|=\frac{2\epsilon_n}{|\lambda| k_0}\leq \epsilon_n$. This proves that $X$ is a strong measure zero set.
\end{proof}

% Combining with the independence of the Borel conjecture in \textsf{ZFC}, we obtain the following conjecture, as discussed in the introduction.
\begin{corollary}\label{cor:1.indep}
    Conjecture \ref{conj:gallagher-lai-weber} is independent of \textsf{ZFC}.
\begin{proof}
    This follows from the fact that the Borel conjecture is independent of \textsf{ZFC}.
\end{proof}
\end{corollary}

\begin{remark} \label{rmk:1} In addition to answering Conjecture \ref{conj:gallagher-lai-weber}, Theorem \ref{thm:R-top-univ} gives a simple and direct proof of known results about topologically universal sets.
\begin{enumerate}
\item It is now easy to see that countable sets are topologically universal. One way to see this is to recall that countable sets are strong measure zero sets (almost by definition). Another way is to directly use Lemma \ref{lem:3}, and note that if $X\subset \R$ is a countable set and $M\subset \R$ is a meager set, then $X+M$ is a countable union of meager sets, which is meager. Since $\R$ is a Baire space by the Baire category theorem, $X+M$ cannot be $\R$, and Lemma \ref{lem:3} finishes the proof. These offer simple alternative ways to deduce the topological universality of countable sets, which was first noted in \cite{gallagpher-lai-weber}.
\item Perfect sets are not topologically universal since perfect sets are not strong measure zero sets \cite[Corollary 8.1.5]{bartoszynski-Judah}. This recovers the previous result \cite[Theorem 1.4]{gallagpher-lai-weber} (for $\R^d$ with $d>1$, we prove it in the next section).
% As noted in \cite{gallagpher-lai-weber}, this implies that all subsets of $\R$ with the perfect set property is not topologically universal.
\item Topologically universal sets have zero Hausdorff dimension, since strong measure zero sets have zero Hausdorff dimension. This recovers a result in \cite[page 3]{gallagpher-lai-weber}.
\item The set of all topologically universal sets form a $\sigma$-ideal. This follows from the fact that strong measure zero sets form a $\sigma$-ideal \cite[Lemma 8.1.2]{bartoszynski-Judah}. The fact that topologically universal sets form a $\sigma$-ideal is a nontrivial statement. Indeed, it is not hard to see that measure universal sets cannot form a $\sigma$-ideal because finite sets are measure universal while we know some countable sets are not measure universal. 
% ?? Maybe spell out? e.g. unbounded countable sets are not measure universal?
% ?? Maybe add uniformly continuous images of topologically unviersal sets are topologically universal?
\end{enumerate}
\end{remark}

We elaborate on Remark \ref{rmk:1} (2). Recall that a set $X$ has the {\bf perfect set property} if $X$ contains a perfect subset. It is well known that every Borel set (more generally, every analytic set) contains a perfect subset \cite[Exercise 14.13]{Kec95} (for original references, see \cite{Luzin-note-on-suslin} as explained in \cite[page 3]{Foreman-Kanamori-handbook-of-set-theory}.) Thus, Borel sets are not topologically universal since they contain perfect sets, which are themselves not topologically universal. Bernstein sets are examples of uncountable subsets of the real line that do not have the perfect set property:
\begin{definition}
    A subset $\mathcal{B}$ of $\R$ is a {\bf Bernstein set} if for every perfect subset $P$ in $\R$, $\mathcal{B}\cap P \neq \varnothing$ and $\mathcal{B}^c \cap P \neq \varnothing$.
\end{definition}
Uncountable Bernstein sets can be constructed in \textsf{ZFC}, using the axiom of choice (see \cite[Example 8.24]{Kec95}). These render the following question in \cite[page 5]{gallagpher-lai-weber} by Gallagher, Lai, and Weber interesting : Are Bernstein sets topologically universal? We give a negative answer to this question in the next example.

\begin{example}[Bernstein sets] Clearly, Bernstein sets are not strong measure zero sets. This is because strong measure zero sets are Lebesgue measurable (since strong measure zero sets have zero Lebesgue measure), but Bernstein sets are not Lebesgue measurable \cite[Theorem 5.4]{Oxt80}. By Theorem \ref{thm:R-top-univ}, Bernstein sets are not topologically universal. This can alternatively be seen from the fact that topologically universal sets have zero Hausdorff dimension, and thus, it has zero Lebesgue measure implying that it must be Lebesgue measurable. 
\end{example}

\section{Generalization to Locally Compact Polish Groups}\label{section:3}
In this section, we generalize Theorem \ref{thm:R-top-univ} and Corollary \ref{cor:1.indep} from $\R$ to locally compact Polish groups. As a special case, we resolve this topological variant of the Erd\H{o}s similarity conjecture for $\R^d$, giving a full answer to \cite[Conjecture 1.9]{gallagpher-lai-weber}.

Recall that a Polish space is a separable completely metrizable topological space. A Polish group $(\G,\cdot)$ is a topological group that is a Polish space. We denote the identity element on $\G$ by ${\mathbf 1}_\G$. By Birkhoff–Kakutani theorem, every Polish group admits a compatible metric $d$ that is left-invariant, but not every Polish group admits a compatible metric that is simultaneously left-invariant and complete. Every locally compact Polish group admits a compatible metric that is left-invariant and complete.
% ?? Name for the locally compact Polish group part?

We now introduce the affine transformation in locally compact Polish groups by generalizing the definition of affine transformations in Euclidean space. By Haar's theorem, for every locally compact Polish group, there is a unique countably additive nontrivial Haar measure upto positive multiplicative constant.

\begin{definition}
Let $\G$ be a locally compact Polish group.
\begin{enumerate}
\item The \textbf{automorphism group} of $\G$, denoted by $\text{Aut}(\G)$, is the collection of all bijective Haar measurable functions from $\G$ to $\G$  that preserve the group operation. 
\item An \textbf{affine transformation} $f$ is a bijection from $\G$ to $\G$ such that there exists $\varphi \in \text{Aut}(\G)$ and $h\in\G$ such that $f(g)=h\cdot \varphi(g)$ for all $g\in\G$.
    \label{affine-topological-group}
\item A subset $X$ of $\G$ is \textbf{topologically universal} in $\G$ if for every dense $G_\delta$ subset $G$ of $\G$ (or equivalently, for every comeager subset of $\G$, by Lemma \ref{lem:3.4} below), there exists an affine map $f:\G\to\G$ such that $f(A)\subset G$.
\end{enumerate}
\end{definition}

\begin{example}
    For $\G=\R^n$, $\text{Aut}(\R^n)=\mathsf{GL}(n,\R)$ and the locally compact Polish group affine transformation coincides with the usual affine transformation on $\R^n$.
\end{example}

% ?? Maybe add explanation: note that measurability is important?

\begin{remark} In this section, we will use the following facts.\label{rmk:2}
    \begin{enumerate}
\item If $\varphi$ is an affine transformation in a locally Polish group, then $\varphi^{-1}$ is also an affine transformation, since an injective Borel measurable function maps Borel sets to Borel sets \cite[Corollary 15.2]{Kec95}
% {\color{red} add a reference}.
\item An automorphism on a locally compact Polish group is a homeomorphism. This is because every Haar measurable homomorphism between locally compact Polish groups is continuous, which is a result of Weil \cite[page 50]{weil-1951}. For a recent generalization of this result to Polish groups that are not necessarily locally compact, see \cite{rosendal}.
\end{enumerate}
\end{remark}

Now, Conjecture \ref{topological-erdos-similarity-conjecture-r} can be naturally generalized to locally compact Polish groups.
\begin{conjecture}[]\label{conj:3}
    There is no uncountable topologically universal set in any locally compact Polish group.
\end{conjecture}

The main goal of this section is to show that the above conjecture is independent of \textsf{ZFC}. We start with the following analogue of Lemma \ref{lem:2}.

% Note that the equivalence in substituting meager sets in the definition of topologically universal set follows from the proof of Proposition \ref{cor-dense-g-delta-comeager} and the Baire Category theorem, which implies that locally compact Polish groups are Baire spaces.
\begin{lemma}\label{lem:3.4}
A subset $X$ of a locally compact Polish group $\G$ is topologically universal if and only if $X$ is universal in the family of all comeager sets in $\G$.
\begin{proof}
    Since Polish spaces are Baire spaces by Baire Category Theorem, Lemma \ref{lem:1} states that every dense $G_\delta$ subset of $\G$ contains a comeager subset of $\G$ and every comeager subset of $\G$ contains a dense $G_\delta$ subset of $\G$. This is sufficient to prove the statement by the same logic as in Lemma \ref{lem:2}.
\end{proof}
\end{lemma}

Next, we need the generalizations of strong measure zero sets.
\begin{definition}[]
\begin{enumerate}
\item A metric space $X$ is said to be \textbf{strong measure zero} if for every sequence $(\epsilon_n)_{n=1}^\infty$ of positive reals, there exists a sequence $(X_n)_{n=1}^\infty$ of subsets of $X$ such that $\text{diam}(X_n)<\epsilon_n$ and $X=\bigcup_{n=1}^\infty X_n$.
% We denote the family of all strong measure zero sets in a metric space $X$ by $\smz(X)$.
\item A subset $X$ of a topological group $\G$ is called \textbf{Rothberger bounded} if for each sequence $(U_n)_{n\in\N}$ of neighborhoods of ${\mathbf 1}_\G$, there exists a sequence $(g_n)_{n\in\N}$ such that $X\subset \bigcup_{n\in\N} \left(g_n\cdot U_n\right)$.

\end{enumerate}
\end{definition}

% ?? add something like we narrow our attention to gain control or obtain desirable spce or appropriate space for the generalization to work.

We first restrict our attention to Polish groups since the notion of Rothberger boundedness coincides with the notion of strong measure zero sets of any compatible metric. This is a consequence of a result by Fremlin \cite[Corollary 534G]{fremlin}.
\begin{theorem*}{C}[Fremlin]\label{thm:C}
    If a topological group $\G$ is a Polish group, then a subset $X$ of $\G$ is Rothberger bounded if and only if $X$ is of strong measure zero in $\G$ with respect to every left-invariant metric on $\G$ that induces the topology of $\G$. 
\end{theorem*}
Hence, we make the following definition.
\begin{definition}[]\label{def:strong-measure-zero-subset-on-polish-group}\index{}
A subset of a Polish group is a \textbf{strong measure zero set} if it is Rothberger bounded.
\end{definition}

The cardinality of strong measure zero subsets of Polish groups and strong measure zero metric spaces are closely related to the cardinality of strong measure zero sets in the real line, as shown in the following theorem.

% Borel conjecture in the real line has analogous equivalent statements for strong measure zero metric spacees and strong measure zero subsets of Polish groups.

% For a Polish group $\G$, the \textsf{BC} (in metric space) now shows that the existence of an uncountable set of strong measure zero in $\G$ is independent of \textsf{ZFC}. From now on, \textsf{BC} may be referred to any context of $\R$, metric spaces, or Polish groups.

\begin{theorem*}{D}\label{thm:D}
The followings are equivalent.
\begin{enumerate}
\item Borel conjecture (every strong measure zero subset of $\R$ is countable).
\item Every strong measure zero metric space is countable.
\item Every strong measure zero subset of a Polish group is countable.
\end{enumerate}
\begin{proof}
   % {\color{red} this theorem looks sloppy, may need to explain in more details} 
   $(1\Rightarrow 2)$ This is a result by Carlson \cite[Theorem 3.2]{carlson}, which shows that if $X$ is an uncountable strong measure zero metric space, then there exists a Lipschitz embedding from $X$ into $\R$. Using an elementary analysis argument,  we can see that images of strong measure zero sets under uniformly continuous functions are strong measure zero sets. As Lipschitz embedding is uniformly continuous, the proof is complete.
   
   % ({\color{red} this is a little wierd. Why results on ${\mathbb R}$ deduce a result on all metric spcaes}).
   $(2\Rightarrow 3)$ Suppose every strong measure zero metric space is countable. Let $X$ be a strong measure zero subset of a Polish group $\G$. By Definition \ref{def:strong-measure-zero-subset-on-polish-group}, $X$ is Rothberger bounded. By Theorem \ref{thm:C}, for any left-invariant metric $d$ on $\G$ that induces the topology of $\G$, $X$ is a strong measure zero metric space. By assupmtion, $X$ is countable.
   % ({\color{red} it is better to say what metric you are using})
   
   $(3\Rightarrow 1)$ This is true since $\R$ is a Polish group.
\end{proof}
\end{theorem*}

Recall that the Galvin-Mycielski-Solovay (Theorem \ref{thm:GMS}) provides a profound characterization of strong measure zero sets in $\R$ in terms of sumsets. A natural question is the following: in generalizing strong measure zero sets, to what extent of spaces do we have an analogue of Galvin-Mycielski-Solovay theorem? It turns out that the answer is locally compact Polish groups, as described in the following theorem by Kysiak in \cite{kysiak} and independently by Fremlin in \cite{fremlin}. For another source of proof, see \cite[Theorem 3.5]{hrusak-zindulka}.

\begin{theorem*}{E}[Kysiak-Fremlin]\label{thm:gms-locally-compact-polish-group}
    If $\G$ is a locally compact Polish group, a subset $X$ of $\G$ is a strong measure zero set if and only if $X\cdot M\neq \G$ for each meager set $M$ in $\G$.
\end{theorem*}
The Galvin-Mycielski-Solovay theorem (Theorem \ref{thm:gms-locally-compact-polish-group}) does not hold for general Polish groups. For example, in \cite[Theorem 3.10]{hrusak-2016}, it was shown that, assuming $\text{cov}(\mathcal{M})=\mathfrak{c}$, the Baer-Specker group $\Z^\N$ is an example of a Polish group that that is not locally compact and that does not satisfy the Galvin-Mycielski-Solovay theorem. Based on these observations, Hrušák and Zindulka conjectured that if we assume $\mathsf{CH}$, then the Galvin-Mycielski-Solovay theorem holds in a Polish group $\G$ and only if $\G$ is locally compact \cite[Conjecture 3.6]{hrusak-zindulka}.

We now prove a sumset characterization of topologically universal sets in locally Polish groups, namely, a generalization of Lemma \ref{lem:3} from $\R$ to any locally compact Polish group.
\begin{lemma}
    Let $\G$ be a locally compact Polish group. A subset $X$ in $\G$ is topologically universal if and only if for every meager set $M$ in $\G$, there exists $\varphi\in\mathrm{Aut}(\G)$ such that $X\cdot\varphi^{-1}(M)\neq \G$ 
    \label{sumset-top-group}
\end{lemma}
\begin{proof}
The proof is analogous to that of Lemma \ref{lem:3}. We first demonstrate the only if part. Suppose $X$ is a topologically universal subset of $\G$. Let $M$ be a meager subset of $\G$. Then, since group inversion is a homeomorphism and homeomorphism maps meager sets to meager sets, $M^{-1}=\{x^{-1}:x\in M\}$ is a meager subset of $\G$. By Lemma \ref{lem:3.4}, there exists $\psi\in\text{Aut}(\G)$ and $h\in \G$ such that $h\cdot \psi (X) \subset (M^{-1})^c$, i.e., $(h\cdot \psi(X)) \cap M^{-1}=\varnothing$. This implies that for all $x\in X$ and $y\in M$, $h\cdot \psi(x) \neq y^{-1}$, i.e., $x\cdot(\psi^{-1}(y))\neq \psi^{-1}(h^{-1})$. Hence, $\psi^{-1}(h^{-1})\not\in X\cdot \psi^{-1}(M)$, i.e., $X\cdot \psi^{-1}(M)\neq \G$. By taking $\varphi=\psi^{-1}\in \text{Aut}(\G)$, we have $X\cdot \varphi(M)\neq \G$.

Next, we show the if part. Let $X$ be a subset of $\G$ and suppose that for every meager set $N$ in $\G$, there exists $\psi\in\text{Aut}(\G)$ such that $X\cdot \psi^{-1}(N)\neq \G$. Let $M$ be a meager subset of $\G$. Again, since group inversion is a homeomorphism and homeomorphism maps meager sets to meager sets, $M^{-1}$ is a meager subset of $\G$. Thus, there exists $\psi\in\text{Aut}(\G)$ such that $X\cdot \psi^{-1}(M^{-1})\neq \G$. This implies that there exists $g\in \G$ such that for all $x\in X$ and $y\in M$, $x\cdot \psi(y^{-1})\neq g$, i.e., $\psi^{-1}(g^{-1})\cdot \psi^{-1} (x) \neq y$. Hence, by letting $\varphi=\psi^{-1}\in \text{Aut}(\G)$ and $h=\psi^{-1}(g^{-1})$, we have $(h\cdot \varphi(X))\cap M = \varnothing$, i.e., $h\cdot \varphi(X)\subset M^c$. Therefore, $X$ is universal in comeager subsets of $\G$, and by Lemma \ref{lem:3.4}, $X$ is topologically universal.
\end{proof}

We can now classify all topologically universal sets in a locally compact Polish group in analogy with Theorem \ref{thm:R-top-univ}.
\begin{theorem}
    \label{classification-group}
    A subset in a locally compact Polish group $\G$ is topologically universal if and only if it is a strong measure zero set.
\end{theorem}
\begin{proof}
    We first prove the if part. Suppose $X$ is a strong measure zero set in $\G$. Then, by Theorem  \ref{thm:gms-locally-compact-polish-group}, for every meager subset of $\G$, we have $X\cdot M\neq \G$. By choosing $\varphi$ to be the identity map $\mathrm{id}_{\G}$ on $\G$ in Lemma \ref{sumset-top-group}, $X$ is topologically universal in $\G$.

    \medskip
    
    Next, we prove the only if part. Suppose $X$ is topologically universal. We are going to show that $X$ is Rothberger bounded.  Let $(U_n)_{n\in\N}$ be a sequence of neighborhoods of ${\bf 1}_\G$. Since $\G$ is a Polish group, there exists a countable neighborhood basis $\mathcal{B}=\{B_n: n\in \N\}$ of ${\mathbf 1}_\G$. Since $\G$ is separable, letting $(g_n)_{n=1}^\infty$ be an enumeration of a dense subset of $\G$, $\bigcup_{n=1}^\infty g_n\cdot U_n$ is a dense open subset of $\G$. Define
    \begin{align*}
        G=\bigcap_{k=1}^\infty \bigcup_{n=1}^\infty g_n\cdot \left(B_k \cap U_n\right).
    \end{align*}
    % \begin{align*}
    %     G=\bigcap_{k=1}^\infty \left(B_k \bigcap \left(\bigcup_{n=1}^\infty g_n U_n \right)\right).
    % \end{align*}
    Since $G$ is a dense $G_\delta$ set, there exists an affine map $f=h \cdot \varphi$ on $\G$ for some $\varphi\in \mathrm{Aut}(\G)$ and $h\in\G$ such that
    \begin{align*}
        f(X)=h \cdot \varphi(X) \subset G.
    \end{align*}
    This implies that
    \begin{align*}
        X \subset \bigcap_{k=1}^\infty \bigcup_{n=1}^\infty \left[\left(\varphi^{-1}(h^{-1}\cdot g_n)\right)\cdot \left(\varphi^{-1} (B_k) \bigcap \varphi^{-1}(U_n)\right)\right].
    \end{align*}
    Since a Haar measurable automorphism in a locally compact Polish group is a homeomorphism, $\varphi$ is a homeomorphism (Remark \ref{rmk:2}). Since homeomorphisms preserve neighborhood bases, $\{\varphi^{-1}(B_k):k\in\N\}$ is a neighborhood basis of ${\bf 1}_\G$. Thus, there exists $k\in\N$ such that $\varphi^{-1}(B_k) \subset U_n$. Hence,
    \begin{align*}
        X &\subset \bigcup_{n=1}^\infty \left[\left(\varphi^{-1}(h^{-1} \cdot g_n)\right)\cdot \left(\varphi^{-1} (B_k) \bigcap \varphi^{-1}(U_n)\right)\right]\\
        &\subset \bigcup_{n=1}^\infty \left(\varphi^{-1}(h^{-1}\cdot g_n)\right)\cdot U_n = \bigcup_{n=1}^\infty g'_n \cdot U_n,
    \end{align*}
    where $g'_n=\varphi^{-1}(h^{-1}\cdot (g_n))$. Therefore, $X$ is a strong measure zero set in $\G$.
\end{proof}

The following corollary, which states that Conjecture \ref{conj:3} is independent of \textup{\textsf{ZFC}}, is now immediate. 

\begin{corollary}
    The existence of an uncountable strong measure zero set in a locally compact Polish group is independent of \textup{\textsf{ZFC}}.
    \begin{proof}
        This directly follows from Theorem \ref{classification-group} and the Borel conjecture for locally compact Polish groups (Theorem \ref{thm:D}).
    \end{proof}
\end{corollary}

\section{Measure-Category Duality and Open Questions}\label{section:4}
\subsection{Measure-Category duality.} In the  literature of descriptive set theory and mathematical logic, several results in measure theory are transferred to analogous results in category (meaning Baire category), and vice versa, by the notion of measure-category duality.
% by the existence of an Erd\H{o}s-Sierp\'{n}ski mapping, assuming the continuum hypothesis (\textsf{CH}), which exchanges ``meager sets" and ``measure zero sets, as stated in the following theorem. \cite{Oxtoby-1980}.
\begin{theorem*}{K}[Erd\H{o}s-Sierpi\'{n}ski]\label{thm:erdos-sierpinski-mapping}
    Assuming the continuum hypothesis (\textup{\textsf{CH}}), there exists a bijection $f:\R\to\R$ such that $f\circ f$ is the identity map on $\R$ and $X$ is a Lebesgue measure zero subset of $\R$ if and only if $f(X)$ is a meager subset of $\R$.
\end{theorem*}

% ?? Steinhaus result for instance
% ?? Principle as theorem?

The mapping in Theorem \ref{thm:erdos-sierpinski-mapping} is referred to as the Erd\H{o}s-Sierp\'{n}ski mapping. The measure-category duality is a principle that suggests if a statement that is formulated in terms of Lebesgue measure zero sets and/or meager sets holds true, then its dual statement, which is obtained by replacing Lebesgue measure zero sets with meager sets in the statement and vice versa, also holds true. It is noteworthy that although the dual statement is easy to obtain, the Erd\H{o}s-Sierpi\'{n}ski mapping cannot be used as a direct proof, thus often requiring a different proof. This is because the Erd\H{o}s-Sierpi\'{n}ski mapping requires \textsf{CH} and it does not preserve addition \cite{Kysiak-another-note-on-duality}, which implies that it cannot directly prove any dual statements involving addition. Despite this, this phenomenon of measure-category duality has been used in various situations \cite{Oxtoby-1980}.

\medskip

Until now, we have classified all topologically universal sets and resolved the topological variant of the Erd\H{o}s similarity conjecture posed by Gallagher, Lai, and Weber in $\R$ and in locally compact Polish groups. In this section, we use the principle of duality between measure and category to translate our results in the topological (category) setting to the measure-theoretic setting.

\subsection{Full measure universality.} We begin by introducing the dual notion of topologically universal sets.
\begin{definition}
    A subset of $\R$ is {\bf full measure universal} if it is universal in the family of all Lebesgue measurable sets with full measure, where a set $E$ has full measure if $\R\setminus E$ has zero Lebesuge measure.
\end{definition}

Erd\H{o}s-Sierp\'{n}ski mapping suggests that measure zero sets and meager sets can be regarded as dual objects under $\mathsf{CH}$. Taking the complement, sets with full measures and comeager sets are dual to each other. Hence, we derive the following conjecture as the dual of Conjecture \ref{conj:gallagher-lai-weber}.
\begin{conjecture}\label{conj:full-measure-erdos} There is no uncountable full measure universal subset in $\R$.
%\begin{enumerate}
 % \item 
  %\item     There is no perfect set that is full measure universal in $\R$.
  %  \end{enumerate}
\end{conjecture}
%begin{conjecture}[??]\label{??}
% ?? Or say that it will anwer 4.2 for perfect sets.
%\end{conjecture}

The dual notions of strong measure zero sets and the Borel conjecture are as follows.
\begin{definition}
    A subset $X$ of $\R$ is called {\bf strongly meager} if $X+M\neq \R$ for every Lebesgue measure zero  set $M$ in $\R$. The {\bf dual Borel conjecture} is the statement that all strongly meager subsets of $\R$ are countable.
\end{definition}

Similar to the Borel conjecture, the dual Borel conjecture is also known to be independent of \textsf{ZFC}. The consistency of the negation of the dual Borel conjecture is due to Pawlikowski and Janusz \cite{pawlikowski-janusz}, where they showed that every Sierpi\'nski set, which is an uncountable subset of $\R$ that intersects every Lebesgue measure zero set countably, is a strongly meager set. Since Sierpi\'nski sets exist assuming \textsf{CH}, this proves the consistency of the negation of the dual Borel conjecture in \textsf{ZFC}. The consistency of the  dual Borel conjecture is due to Carlson. A detailed review of the dual Borel conjecture can be found in \cite[Chapter 8]{bartoszynski-Judah}, and recent generalizations and developments can be found in \cite{scheepers-zindulka}.

\begin{proposition}
    \label{lem:sumset-full-measure}
    A subset $X$ of $\R$ is full measure universal if and only if for every Lebesuge measure zero set $M$ in $\R$, there exists $\lambda\neq 0$ such that $X+\lambda M \neq \R$.
Consequently, strongly meager sets are full measure universal.
   \end{proposition}
 
   \begin{proof}
        The first part of the proof is identical to the proof of Lemma \ref{lem:3} by substituting Lebesgue measure zero sets in place of meager sets.

     To see the second part, let $X$ be a strongly meager set. Then by the definition of strongly meager set, $X$ satisfies the first part of the statement with $\lambda =1$. Thus, $X$ is a full measure universal set in $\R$.
    \end{proof}

Using strongly meager sets, it is natural to pose the following classification conjecture for full measure universal sets, which is the dual of Theorem \ref{thm:R-top-univ}.

\begin{conjecture}\label{conj:full-measure-iff-strongly-meager}
    A subset of $\R$ is full measure universal if and only if it is a strongly meager subset of $\R$.
\end{conjecture}

\begin{remark} We make few remarks on Conjectures \ref{conj:full-measure-erdos} and \ref{conj:full-measure-iff-strongly-meager}.
\begin{enumerate}
%\item ?? perhaps rather than no uncountable set is full measure universal, no perfect set is full measure universal?
\item Note that Conjecture \ref{conj:full-measure-erdos} can also be viewed as the uncountable version of the original Erd\H{o}s similarity conjecture (Conjecture \ref{conj:original-erdos-similarity-conjecture}). To the best of our knowledge, Conjecture \ref{conj:full-measure-erdos} and Conjecture \ref{conj:original-erdos-similarity-conjecture} do not imply each other.

% these two conjectures ({\color{red} which two conjectures?}) are independent of each other (all outcomes are possible {\color{red} do you mean they do not imply each other?}).
\item The validity of Conjecture \ref{conj:full-measure-iff-strongly-meager}  will imply that Conjecture \ref{conj:full-measure-erdos} is independent of \textsf{ZFC} by the fact that dual Borel conjecture is independent of \textsf{ZFC}, which, together with Theorem \ref{thm:R-top-univ}, confirms the principle of duality between measure and category.
\item As mentioned in Remark \ref{rmk:1}, strong measure zero sets form a $\sigma$-ideal. This motivated Prikry to conjecture that strongly meager sets also form a $\sigma$-ideal. However, it turns out that if one assumes \textsf{CH}, then strongly meager sets actually do not form a $\sigma$-ideal, as shown by Bartoszy\'nski and Shelah in \cite{bartoszynski-tomek-shelah}\footnote{This is an example of a subtlety in following the principle of measure-category duality. It often suggests dual results that turn out to be correct, but not always. As shown in this case, strongly meager sets sometimes behave worse than strong measure zero sets.}. We ask the following question: 

\smallskip

{\bf (Qu):} {\it Do full measure universal sets form  $\sigma$-ideals?} 

\smallskip

If the answer is yes, then assuming \textsf{CH}, the set of all strongly meager sets is a strict subset of the set of all strong measure zero sets. This produces a model of \textsf{ZFC} in which Conjecture \ref{conj:full-measure-iff-strongly-meager} is false, proving that the negation of the Conjecture \ref{conj:full-measure-iff-strongly-meager} is consistent in \textsf{ZFC}. We note that we do not yet know if Conjecture \ref{conj:full-measure-iff-strongly-meager} is consistent in \textsf{ZFC}, in which case, Conjecture \ref{conj:full-measure-iff-strongly-meager} would be independent of \textsf{ZFC}.

% If the answer is yes, then assuming \textsf{CH}, the set of all strongly meager sets is a strict subset of the set of all strong measure zero sets, contradicting Conjecture \ref{conj:full-measure-iff-strongly-meager} {\color{red} do you mean the negation of Conjecture \ref{conj:full-measure-iff-strongly-meager} is correct?}. 
% {\color{red} Why do we need to say anything about the original conjecture? It looks very far to say that the original conjecture is independent of ZFC.} This does not disprove Conjecture \ref{conj:full-measure-iff-strongly-meager}, though, as it only implies that the negation of Conjecture \ref{conj:full-measure-iff-strongly-meager} is consistent in \textsf{ZFC}. It might still be the case that Conjecture \ref{conj:full-measure-iff-strongly-meager} is consistent in \textsf{ZFC}. In that case, Conjecture \ref{conj:full-measure-iff-strongly-meager} would be independent of \textsf{ZFC}.

\item The measure-category duality (Theorem \ref{thm:erdos-sierpinski-mapping}) extends to arbitrary locally compact Polish groups. This allows one to define full measure universal sets and strongly meager sets in locally Polish group as the dual of strong measure zero sets via the Galvin-Mycielski-Solovay theorem (Theorem \ref{thm:gms-locally-compact-polish-group}): a set $X$ in a locally compact Polish group $\G$ is full measure universal if it is universal in full Haar measure sets in $\G$, and $X$ is strongly meager if $X+M\neq \G$ for all Haar measure zero subsets $M$ of $\G$.
% However, the literature on generalizations of strongly meager sets on spaces beyond $\R$ is still in infancy. We hope this connection with Erd\H{o}s similarity conjecture might provide further motivation for research in this direction.
\end{enumerate}
\end{remark}

\subsection{Perfect sets} Unfortunately, Conjecture \ref{conj:full-measure-iff-strongly-meager} is open even for perfect sets.   Erd\H{o}s, Kunen, and Mauldin \cite{erdos-kunen-mauldin} showed  that perfect sets are not strongly meager. 
\begin{theorem}[Erd\H{o}s-Kunen-Mauldin]\label{thm:erdos-kunen-mauldin}
    For every perfect set $P\subset\R$, there exists a set $M\subset\R$ of zero Lebesgue measure such that $P+M=\R$.
\end{theorem}
  Erd\H{o}s-Kunen-Mauldin theorem does not imply that perfect sets are not full measure universal because we do not know if the same $M$ works for all $\lambda\ne 1$ in Lemma \ref{lem:sumset-full-measure}. Hence, we pose the following generalization of the Erd\H{o}s-Kunen-Mauldin Theorem as a conjecture, which would imply that all perfect sets are not full measure universal using Proposition \ref{lem:sumset-full-measure}.
\begin{conjecture}\label{conj:full-meausre-perfect}
    Let $P\subset \R$ be a perfect set. Then, there exists a set $M\subset \R$ with Lebesgue measure zero such that $P+\lambda M=\R$ for all $\lambda\neq 0$. 
\end{conjecture}

Conjecture \ref{conj:full-measure-iff-strongly-meager} and Conjecture \ref{conj:full-meausre-perfect} will resolve a stronger version of the Erd\H{o}s similarity conjecture for perfect sets. Finally, we conclude this section with a sufficient condition for Conjecture \ref{conj:full-meausre-perfect}, which only require robustness under arbitrarily small amount of scaling.

\begin{proposition}
    Suppose that for all compact perfect set $P$, there exist $\epsilon_0>0$, $\delta>0$ and a compact set $K$ with zero Lebesgue measure such that $P+ t K$ contains an open interval of length $\epsilon_0$ for all $t\in (1-\delta,1]$. Then, there exists an $F_{\sigma}$ set $M$ satisfying Conjecture \ref{conj:full-meausre-perfect}.  
\end{proposition}

\begin{proof}
    Let $1<a<\frac{1}{1-\delta}$. Then, all consecutive intervals $(a^j(1-\delta), a^j]$ and $(a^{j+1}(1-\delta), a^{j+1}]$ overlap. Moreover, as $a>1$, 
    \begin{equation}\label{dilation}
         \bigcup_{j\in \Z} (a^j(1-\delta), a^j] = (0,+\infty).
    \end{equation}
    Define 
    $$
    M = \bigcup_{j\in\Z} \bigcup_{k\in\Z}a^j \left(K+ {\epsilon_0}k\right).
    $$
    Then $M$ clearly has zero Lebesgue measure. Note that by (1), for each $\lambda>0$, we can find $j\in \Z$ so that $t: = a^{-j}\lambda\in (1-\delta, 1]$. Let $I$ be the open interval inside $P+t K$ with length $\epsilon_0$. Then, since $\lambda a^{-j}\le 1$, we have
    $$
    P+\lambda M\supset P+\lambda\bigcup_{k\in\Z}a^{-j} \left(K+ {\epsilon_0}k\right)  \supset \bigcup_{k\in\Z}\left(I+ \epsilon_0 t k\right) = \R.
    $$
    \end{proof}
Such compact set $K$ can easily be constructed if $P$ has positive Newhouse thickness, in which case, one can use the Newhouse gap lemma. This recovers the result that Cantor sets with positive Newhouse thickness are not full-measure universal in \cite[Theorem 1.5]{gallagpher-lai-weber}. We were also informed by P. Shmerkin that the same is also true if $P$ has positive Hausdorff dimension using a probabilistic argument in \cite[Theorem 13.1]{SV2015} (an argument is outlined in \cite{Jung-Lai-Mooroogen}). We do not have a definitive answer to Conjecture \ref{conj:full-meausre-perfect} if Cantor set has both Hausdorff dimension and Newhouse thickness zero. 

\noindent{\bf Acknowledgments} The authors thank Daeseok Lee, Krystal Taylor, Alex McDonald, and Matthew Foreman for helpful discussions.
% {\color{red} Should we add also Krystal and Alex?}

% for helpful discussions on simplifying the initial proof of Theorem \ref{thm:R-top-univ}. The authors would also like to thank Matthew Foreman for clarifying the Borel conjecture on locally compact Polish groups.

\bibliographystyle{amsplain}
\bibliography{refs}

\end{document}